\documentclass[11pt]{amsart}
\input epsf
\usepackage{latexsym}
\usepackage[usenames,dvipsnames]{color}
\usepackage[colorlinks=true,linkcolor=blue,citecolor=BrickRed,urlcolor=RoyalBlue]{hyperref}
            
\usepackage{amssymb,amsmath,amsthm,amscd, amssymb,
graphicx, enumerate, verbatim, mathrsfs, color, url, bm, hyperref, parskip}
\usepackage[all]{xy}\usepackage[labelformat=empty]{caption}

\usepackage[left]{lineno}

\def\Diff{\operatorname{Diff}}

\def\Int{\operatorname{Int}}

\numberwithin{equation}{section}
\newtheorem{cor}[equation]{Corollary}
\newtheorem{lem}[equation]{Lemma}
\newtheorem{prop}[equation]{Proposition}
\newtheorem{thm}[equation]{Theorem}

\newtheorem{quest}[equation]{Question}

\newtheorem{Example}[equation]{Example}

\newtheorem{remark}[equation]{Remark}
\newenvironment{rmk}{\begin{remark}\rm}{\end{remark}}
\def\co{\colon\thinspace}

\newcommand{\Iso}{\mbox{Iso}}

\newcommand{\e}{\varepsilon}

\def\a{\alpha}

\def\b{\beta}

\def\d{\partial}

\def\i{\iota}

\def\cf{\mathfrak c}
\def\R{\mathbb{R}}

\def\Z{\mathbb{Z}}

\def\Z{\mathbb{Z}}
\def\S1{\bf S^1}

\newcommand{\D}{\bm{\mathcal D}}
\newcommand{\X}{\mathcal X}
\newcommand{\E}{\bm{\mathcal E}}
\newcommand{\Se}{\bm{\mathcal S}}
\newcommand{\Ze}{\bm{\mathcal Z}}

\DeclareMathOperator{\Emb}{Emb}

\makeatletter
\def\equalsfill{$\m@th\mathord=\mkern-7mu
\cleaders\hbox{$\!\mathord=\!$}\hfill
\mkern-7mu\mathord=$}
\makeatother

\begin{document}

\abovedisplayskip=6pt plus3pt minus3pt
\belowdisplayskip=6pt plus3pt minus3pt

\title[Centers of disks in Riemannian manifolds]{\bf Centers of disks in Riemannian manifolds}

\date{\today \,(Last Typeset)}


\thanks{The research of M. G. was supported in part by NSF grants DMS-1308777 and DMS-1711400.}

\author{Igor Belegradek}
\address{Igor Belegradek\\ School of Mathematics\\ Georgia Institute of
Technology\\ Atlanta, GA, USA 30332\vspace*{-0.1in}}
\email{ib@math.gatech.edu\vspace*{-0.1in}}
\urladdr{www.math.gatech.edu/$\sim$ib}

\author{Mohammad Ghomi}
\address{Mohammad Ghomi\\School of Mathematics, Georgia Institute of Technology,
Atlanta, GA , USA 30332\vspace*{-0.1in}}
\email{ghomi@math.gatech.edu\vspace*{-0.1in}}
\urladdr{www.math.gatech.edu/$\sim$ghomi}

\begin{abstract}
We prove the existence of a center,  or continuous selection of a point,
in the relative interior of $C^1$ embedded $k$-disks 
in Riemannian $n$-manifolds. 
If $k\le 3$ the center can be made equivariant 
with respect to the isometries of the manifold, and under mild assumptions 
the same holds for $k=4=n$. By contrast, for every $n\ge k\ge 6$ there are examples where
an equivariant center does not exist.
The center can be chosen to agree with any of the classical centers
defined on the set of convex compacta in the Euclidean space. 
\end{abstract}
\maketitle
\thispagestyle{empty}

\section{Introduction}
\label{sec: intro}

There are several distinguished points inside a compact convex subset of
the Euclidean space, 
see~\cite{KMT}, \cite[Chapter 12]{moszynska:book}, and~\cite[Section 5.4.1]{Sch-book}.
In this paper we investigate if there is a geometrically
meaningful point inside every embedded disk in a Riemannian manifold.
Any such point should depend continuously on the disk and be equivariant under 
isometries.

To set the stage
let $G$ be a subgroup of the isometry group 
of a smooth (i.e., $C^\infty$) connected Riemannian manifold $M$,
and $\X(M)$ be the space of compact connected $C^1$ embedded submanifolds of $M$ 
equipped with the $C^1$ topology (see Section~\ref{sec: space of sbmfls}). 
For a subspace $X$ of $\X(M)$, a {\em center\,}
is a continuous map $\cf\co X\to M$ such that $\cf(D)$ lies in the relative interior of $D$
for every $D\in X$. 
If in addition $X$ is \mbox{$G$-invariant} and $\cf$ is $G$-equivariant, we say that $\cf$
is {\em $G$-equivariant}, which simply means that $\cf(g D)=g\hspace{.3pt}\cf(D)$ for all $D\in X$, $g\in G$. 
We investigate the existence of $\cf$ for a given triple $(M,G,X)$.

The best-known examples of a center are the center of mass
and the Steiner point, which we call
the {\em classical centers\,}. They are $\Iso(\R^n)$-equi\-var\-i\-ant centers
on the space of convex compacta in $\X(\R^n)$, where
$\Iso(\R^n)$ is the group of Euclidean isometries, see~\cite{Sch-book}.

To state our findings, 
for any $D\in \X(M)$ let $G_D$ be the subgroup of $G$ which maps $D$ to itself, and 
$D^{G_D}$ be the points of $D$ which are fixed by $G_D$ (see Section \ref{sec: groups}). 
Note that if $\cf$ is $G$-equivariant, then $\cf(D)\in D^{G_D}$, and in particular,
no $G$-equivariant center exists if $D^{G_D}=\emptyset$ for some $D\in X$.
Our main result is as follows:

\begin{thm}\label{thm: intro main}
The space of submanifolds $D\in\X(M)$  
such that $D^{G_D}$ is contractible
admits a $G$-equivariant center. 
Furthermore, this center can be chosen to agree
with any given $G$-equivariant center defined on a closed subset 
$A$ of $\X(M)$.
\end{thm}

An important example of $A$ is the set of convex compacta in $\R^n$
equipped with one of the classical centers mentioned above. In this case
our results are summarized in Corollary~\ref{cor: all for Rn} below.

We prove Theorem~\ref{thm: intro main} by reducing the problem to finding sections of 
certain bundles with $D^{G_D}$ as fibers, and then fitting the sections together
to define a global center. The construction involves some choices,
and the resulting center is not canonical. 

Let us review some natural conditions under 
which $D^{G_D}$ is contractible, which yield the 
following corollaries of Theorem~\ref{thm: intro main}. 
It is to be understood that {\em each of these corollaries has a relative version,\,}
i.e., the corresponding center can be chosen to agree
with any given $G$-equivariant center defined on a closed subset of
$\X(M)$.  
First note that if $G$ is trivial, then $D^{G_D}=D$. So we obtain:

\begin{cor}
The space of contractible submanifolds $D\in\X(M)$ admits a center. 
\end{cor}

Another condition that ensures contractibility of $D^{G_D}$ 
is that each $D\in\X(M)$ is homeomorphic to a disk of dimension $\le 3$.
This is due to the  fact that any smooth action of a compact 
Lie group on a disk of dimension $\le 3$ is smoothly
equivalent to a linear action, which implies that $D^{G_D}$ is a disk.  
For actions on $2$-disks linearity is a standard consequence of the uniformization theorem
(see the beginning of Appendix~\ref{sec: actions on D4}), while 
the $3$-dimensional case is established in~\cite[Theorem B]{KwaSch-S3-linear}. 

Let $\mathfrak{D}^{k\hspace{-.3pt}}(M)$ be the subspace of $D\in\X(M)$ such that 
$D$ is homeomorphic to a disk of dimension $\le k$. 
From the previous paragraph we obtain:

\begin{cor}
\label{cor: disks of dim at most 3}
$\mathfrak{D}^{3\hspace{-.3pt}}(M)$ 
admits a $G$-equivariant center.
\end{cor}

On the other hand, contractibility of $D^{G_D}$ generally fails 
when $D$ is homeomorphic to a $4$-disk, because then 
$D^{G_D}$ may be an acyclic non-contractible $3$-manifold, see Lemma~\ref{lem: acyclic 4d}.
For this case to occur $G_D$ must be generated by an involution
whose action on $D$ reverses the orientation, see Lemmas~\ref{lem:4-disk1} and~\ref{lem:4-disk2}. 
Avoiding this case, we obtain  the existence of centers in dimension $4$:

\begin{cor}
\label{cor: 4-disk orient}
If $\dim(M)=4$ and the $G$-action on $M$ is orientation-preserving, then 
$\mathfrak{D}^{4\hspace{-.3pt}}(M)$   
admits a $G$-equivariant center.
\end{cor}
 
Another result for $4$-manifolds can be obtained by observing that
no compact acyclic non-contractible $3$-manifold smoothly embeds into $S^3$
due to the topological Schoenflies theorem~\cite{Bro-sch}.
In Lemma~\ref{lem: fixed point set is C1} we note that
$D^{G_D}$ is a $C^1$ submanifold of $M^{G_D}$, the fixed point set of $G_D$ in $M$.
Therefore, a non-contractible $D^{G_D}$ cannot occur if for any involution in $G$
its fixed point set in $M$
smoothly embeds into $S^3$. The latter can be forced by the following
geometric assumptions.

\begin{cor}
\label{cor: 4-disk main}
If either $M=S^4$ with its standard action of $G=O(5)$,
or $M$ is a Hadamard $4$-manifold with isometry group $G$, 
then $\mathfrak{D}^{4\hspace{-.3pt}}(M)$   
admits a  $G$-equivariant center.
\end{cor}

Recall that a {\em Hadamard manifold\,} is a
contractible complete Riemannian manifold of nonpositive sectional curvature,
and the fixed point set of any isometric action on a Hadamard manifold
is diffeomorphic to a Euclidean space, due to the Cartan-Hadamard theorem.

The above $4$-dimensional results depend on establishing that the
fixed point set of a smooth compact Lie group action on the $4$-disk
is either a disk or a compact acyclic \mbox{$3$-manifold}.
We check this in Lemmas~\ref{lem:4-disk1} and~\ref{lem:4-disk2}. 
A key ingredient is that any smooth finite group action on the $4$-disk
fixes a point. This was
established in~\cite[Theorem II.1]{BKS-fixed-pt-D4} modulo an announcement of Thurston
that any non-free smooth finite group action on $S^3$ preserves a round metric.
Thurston's claim was finally proved in~\cite{DinLee-RF-geom} via a Ricci flow arguments,
and his outline was made rigorous in~\cite[Corollary 1.1]{BLP}
for orientation-preserving actions.

In dimension $5$ the situation is unclear, e.g.,  
it does not seem to be known whether every smooth action of a
compact Lie group on a $5$-disk has a fixed point.
The strategy used in Corollary~\ref{cor: 4-disk main} for linear actions on $\R^4$
breaks down one dimension up, because there is a linear involution of $\R^5$
whose fixed point set on some embedded disk is not contractible, see Lemma~\ref{lem: acyclic 5d}.

{\em Starting from dimension $6$ an equivariant center need not exist\,}, namely, 
there are triples $(M,X,G)$ such that $D^{G_D}=\emptyset$ for some smoothly embedded disk
$D\in X$, 
which rules out the existence of a $G$-equivariant center, as we had mentioned earlier.
Indeed, for each \mbox{$n\ge 6$} there is a smooth action of the alternating group $A_5$ on $S^n$ 
with exactly one fixed point, see~\cite{BakMor} and references therein, 
and excising an invariant tubular neighborhood 
of this point yields a fixed point free smooth action of $A_5$ on the $n$-disk.
One can then embed the resulting $A_5$-action on the $n$-disk into 
a fixed point free $A_5$-action on a boundaryless manifold, 
e.g., into $\R^n$  by attaching a collar along the boundary, or into $S^n$
by doubling along the boundary.
 
The Mostow-Palais theorem~\cite[Theorem VI.4.1]{Bre-book} yields a smooth
embedding of the above fixed point free
$A_5$-action on the $6$-disk into an orthogonal $A_5$-action on some $\R^r$.
(With some effort one can obtain an explicit upper bound for 
$r$ but we will not attempt it here because we are unable to determine 
the optimal $r$.) 
As we shall explain at the end of Section~\ref{sec: the proof},
this discussion easily implies the following:

\begin{prop}
\label{prop: no center}
For any $m$, $n$ such that either $m\ge n+r-6\ge r$ or
$m\ge n\ge r$ there is an orthogonal $A_5$-action on $\R^m$
that preserves a smoothly embedded $n$-disk on which $A_5$
acts without a fixed point. In particular,
$\mathfrak{D}^{n\hspace{-.3pt}}(\R^m)\setminus\mathfrak{D}^{n-1\hspace{-.3pt}}(\R^m)$  
does not admit an $O(m)$-equivariant center. 
\end{prop}

At this point it seems worthwhile to summarize our results for the Euclidean space: 

\begin{cor}
\label{cor: all for Rn}
Any $\mathrm{Iso}(\R^n)$-equivariant center on the set of convex compacta in $\R^n$
extends to a center $\cf\co\mathfrak{D}^{n}(\R^n)\to\R^n$ that
is $\mathrm{Iso}(\R^n)$-equivariant on $\mathfrak{D}^{3}(\R^n)$.
Moreover, if $n=4$, then $\cf$ is $\mathrm{Iso}(\R^4)$-equivariant 
on $\mathfrak{D}^{4}(\R^4)$. If $k$ is sufficiently large, 
then $\mathfrak{D}^{k}(\R^n)$ does not admit an
$\mathrm{Iso}(\R^n)$-equivariant center.
\end{cor}

The results of this paper belong to 
the subject of \emph{continuous selections of multivalued mappings}, 
an established branch of topology which was
pioneered by Michael~\cite{michael1, michael2, michael60} and extensively
surveyed in~\cite{RepSem-book, RepSem-II, RepSem-III}. 
Indeed, assigning to each $D\in \mathcal X(M)$ the relative interior of $D^{G_D}$ in $M$ yields
a multivalued map $\mathcal X(M)\to M$ which is lower semicontinuous 
in the sense of~\cite[Definition 0.43]{RepSem-book}.
What we call a $G$-equivariant center is a continuous $G$-equivariant 
selection of this multivalued map over a subset $X$ of $\mathcal X(M)$.
The theory of continuous selections 
implies existence of a non-equivariant center under certain assumptions
on $X\subset\mathcal X(M)$ that tend to come in two flavors: either
every $D\in X$ needs to satisfy a suitable generalized convexity condition,
or $X$ is required to be finite-dimensional (see~\cite{RepSem-book}
for details). 
By contrast, our focus is on equivariant centers, and
 there seems to be no prior work analogous to the results 
of this paper.

The classical centers are
continuous in the Hausdorff topology on the set of convex compacta,
and hence one might expect that this would be a natural topology for $\X(M)$. 
However, in Remark~\ref{rmk: no Hausdorff} we shall see
that Theorem~\ref{thm: intro main} fails when $\X(\R^n)$ is given the Hausdorff
topology. On the other hand, for any $k$ the Hausdorff topology on 
the set of $k$-dimensional convex compacta in $\X(\R^n)$ 
coincides with the $C^0$ topology, and hence we ask: 

\begin{quest}
\label{quest: intro C0 top}
{\rm
Does \textup{Theorem~\ref{thm: intro main}} remain valid when $\X(M)$ is replaced with
the space of $C^0$ submanifolds equipped with the $C^0$ topology?}
\end{quest}

\section{Background on group actions}\label{sec: groups}

Here we 
review the basic facts on Lie group actions used in this work.
Throughout this section $G$ is a Lie group and $X$ is a metrizable space.
A {\em $G$-action\,} on $X$ is a continuous map $a\co G\times X\to G$, written as 
$gx:=a(g,x)$, such that the map $g\to a(g,\cdot )$ 
is a homomorphism between $G$
and the homeomorphism group of $X$. A space with a $G$-action
is a {\em $G$-space}. The {\em orbit space\,} $X/G$
is the set of $G$-orbits with the quotient topology.
A map $f\co X\to Y$ of $G$-spaces is a {\em $G$-map\,} if $f$ is continuous and $f(gx)=gf(x)$ for all $x\in X$, and $g\in G$;
we will also refer to $f$ as {\em $G$-equivariant\,} or just {\em equivariant\,} when $G$ is understood.
For $S\subseteq X$, we let $gS:=\{gx\,:\, x\in S\}$ and use the following notations:
\begin{center}
$G_{\hspace{-.3pt}S}:=\{g\in G\,:\, \text{$gS=S$}\}\,=$ the isotropy subgroup of $S$ in $G$,
\end{center}  
\begin{center}  
$GS:=\{gS\,:\, g\in G\}\,=$ the $G$-orbit of $S$,
\end{center}  
\begin{center}  
$X^G:=\{x\in X\,:\, \text{$gx=x$ for each $g\in G$}\}\,=$
the fixed point set of $G$ in $X$.
\end{center}  
If $S=\{x\}$, then $GS$, $G_S$ are denoted by $Gx$, $G_x$, respectively.
Note that $G_x$ is a closed subgroup of $G$. If $GS=S$, then $S$ is {\em $G$-invariant}.
If $x\in S_x\subseteq X$, then $S_x$ is called a {\em $G_x$-slice at $x$} provided that
$GS_x$ is open in $X$ and there is a $G$-map
$f\co GS_x\to G/G_x$ with $S_x=f^{-1}(G_x)$.

\begin{lem}
\label{lem: preim of slice}
Let $\a\co A\to X$ be a $G$-map of $G$-spaces
with $\a(a)=x$ and $G_a=G_x$. If
$S_x$ is a $G_x$-slice at $x$, then $\a^{-1}(S_x)$
is a $G_a$-slice at $a$.  
\end{lem}
\begin{proof} 
Note that
$\a\circ f\co A\to G/G_x=G/G_a$ is a $G$-map
with $(\a\circ f)^{-1}(G_a)=\a^{-1}(S_x)$. 
For any $g\in G$, $z\in X$, we have $g\a^{-1}(z)=\a^{-1}(gz)$, 
and hence $G\a^{-1}(S_x)=\a^{-1}(GS_x)$. Thus openness of
$GS_x$ implies openness of $G\a^{-1}(S_x)$. 
\end{proof}

A $G$-space $X$ is {\em Palais-proper\,} 
if any $x\in X$ has a neighborhood $V_x$ such
that every $y\in Y$ has a neighborhood $V_y$ for which
$\{g\in G\,:\, gV_x\cap V_y\neq\emptyset\}$ is precompact 
in $G$. For example, if $G$ is compact, then any $G$-space is Palais-proper.
If $G$ is an isometry group of a smooth Riemannian manifold, 
then the $G$-space is Palais-proper~\cite[Theorem I.4.7]{KobNom},
and conversely, any smooth Palais-proper action preserves
a smooth Riemannian  metric~\cite[Theorem 4.3.1]{Pal-slice}.

Note that if $x$ is a point of a Palais-proper $G$-space $X$, then
$G_x$ is compact (because it is closed and precompact in $G$).
A key result established in~\cite[Section 2.3]{Pal-slice} is that
every point $x$ in a Palais-proper $G$-space
is contained in a $G_x$-slice.
In~\cite[Section 2.1]{Pal-slice} one finds the following characterization
of slices at points with compact isotropy subgroups: 

\begin{lem}[\cite{Pal-slice}]
\label{lem: charact of slices}
Let $x\in S\subseteq X$ and suppose that $G_x$ is compact.
Then $S$ is a $G_x$-slice at $x$ if and only if
the following conditions hold:
\vspace{-3pt}
\begin{itemize}
\item[\textup{(a)}]  $gS\cap S\neq\emptyset$ if and only if $g\in G_x$, \vspace{2.5pt}
\item[\textup{(b)}] $S$ is closed in $GS$, and  $GS$ is open in $X$, \vspace{2pt}
\item[\textup{(c)}] there is an open set $O$ such that 
$S\subset O\subset GS$ and $\{g\in G\,:\, gO\cap O\neq\emptyset\}$
is precompact in $G$.\qed
\end{itemize}
\end{lem}

The following lemma summarizes what we need to know about slices.

\begin{lem}\label{lem: implications of slice}
If $S_x$ is a $G_x$-slice at $x$ and $G_x$ is compact, then the following statements hold:
\begin{itemize}
\item[\textup{(i)}]
$G_y\subset G_x$ for every $y\in S_x$.\vspace{2pt}
\item[\textup{(ii)}]
If $W$ is open in a slice $S_x$, then $GW$ is open in $X$. \vspace{2pt}
\item[\textup{(iii)}]
The inclusion induced map $S_x/G_x\to X/G$ 
is an open embedding.\vspace{2pt}
\item[\textup{(iv)}]
Any open $G_x$-invariant neighborhood of $x$ in $S_x$ is a $G_x$-slice at $x$. 
\vspace{2pt}
\item[\textup{(v)}]
Any neighborhood of $x$ in $S_x$ contains an open set that is a $G_x$-slice at $x$.
\vspace{2pt}
\item[\textup{(vi)}] For any neighborhood $U$ of the identity in $G$ there
is a neighborhood $V$ of $x$ in $GS$ such that for each $y\in V$ there is $u\in U$
with $u^{-1}G_yu\subseteq G_x$.  
\end{itemize}
\end{lem}
\begin{proof}
(i) is immediate by Lemma~\ref{lem: charact of slices}(a), while
(ii) is proved in ~\cite[Corollary on p.306]{Pal-slice}. The rest of the items are established as follows:

(iii): The map here is one-to-one by Lemma~\ref{lem: charact of slices}(a). Further, it is
a homeomorphism by (ii) and the defining properties of the quotient topology. 

(iv): Let $W$ be an
open $G_x$-invariant neighborhood of $x$ in $S_x$, and appeal to Lemma~\ref{lem: charact of slices}.
The conditions (a) and (c) are immediate, and (i) implies that $GW$ is open in $X$. 
To see that $W$  is closed in $GW$ 
take $w_i\to gw$, where $g\in G$, $w_i, w\in W$, and note that 
$w, w_i\in S_x$ implies $gw\in S_x$, so $g\in G_x$ and hence $gw\in W$.

(v): Note that $S_x$ has a $G_x$-invariant metric~\cite[Proposition 1.1.12]{Pal-mem},
so any open metric ball  centered at $x$ is $G_x$-invariant, and hence is a slice by (iii).
Any neighborhood of $x$ contains such a ball. 

(vi): Choose $V$ inside $O$ of Lemma~\ref{lem: charact of slices}(c).
Then $G_x$, $G_y$ lie in a compact subgroup of $G$, in which case
a proof can be found in~\cite[Corollary II.5.5]{Bre-book}.
\end{proof}

\section{Space of Submanifolds}
\label{sec: space of sbmfls}

In this section we give a precise definition of the topology on $\X(M)$. Further we  
show that this topology is Palais-proper and induced by a $G$-invariant metric.
For any submanifold $D\in \X(M)$
let $\X_D(M)\subset \X(M)$ be the collection of submanifolds 
which are $C^1$ diffeomorphic to $D$. In other words,
$$
\X_D(M):={\Emb}^1(D, M)/\Diff^1(D),
$$
the space of $C^1$ embeddings of $D$ into $M$ modulo $C^1$ diffeomorphisms of $D$.
We equip $\X_D(M)$ with its standard
$C^1$ topology, which is induced by $C^1(D,M)$, the space of $C^1$ mappings $D\to M$. Thus a pair of submanifolds $A$, $B\in \X_D(M)$ are close
if they admit parametrizations $f$,  $g\in {\Emb}^1(D, M)$ that are $C^1$ close.
Finally we topologize $\X(M)$ as the disjoint union of $\X_D(M)$
where $D$ ranges over 
$C^1$ diffeomorphism classes of submanifolds $D\in \X(M)$. Note that
the obvious $G$-action on $\X(M)$ given by $\a(g, D)=gD$ is effective.

\begin{lem}
\label{lem: properties of XkD(M)}
The following statements hold:
\begin{enumerate}
\item The topology on $\X_D(M)$ is induced by a $G$-invariant metric $d_D$.
\item The $G$-space $\X_D(M)$ is Palais-proper.
\item The orbit space $\X_D(M)/G$ is metrizable.
\end{enumerate}
\end{lem}
\begin{proof}
A convenient way to describe $\X_D(M)$ is to consider 
a smooth $G$-equivariant embedding of $M$ into a Hilbert space $\mathcal H$ 
equipped with some orthogonal $G$-action, see~\cite[Theorem 0.1]{Kaa-emb}.
In general, one cannot equivariantly embed $M$ into a
finite dimensional linear space; this can be done, e.g., 
if both $G$ and $M$ are compact~\cite[Theorem VI.4.1]{Bre-book}.

To prove (1) we define a $G$-invariant metric $d_D$ on $\X_D(\mathcal H)$,
and then restrict it to a smoothly embedded $G$-invariant copy of $M$ in $\mathcal H$.
Define $d_D(A,B)$ as the infimum of 
$\|\a-\b\|_{C^1}$ taken over all $C^1$ embeddings
$\a, \b\in C^1(D,\mathcal H)$ with images $A$, $B$, respectively.
Here the $C^1$ norm is computed using the norm on $\mathcal H$. 
The triangle inequality follows from the one for $\mathcal H$
and properties of the infimum, and the $G$-invariance holds because
the $G$-action on $\mathcal H$ preserves the norm. 
For non-degeneracy note that if say $a\in A\setminus B$,
then $d_D(A,B)$ is bounded below by the distance from $a$ to $B$,
hence $d_D(A,B)=0$ implies $A=B$.

To prove (2) fix $r\in (0,1)$ and let $V_x$ be
the $r$-ball about $x$ in $(\X_D(M), d)$. 
If $A, B\in \X_D(M)$ with $V_A\cap g V_B\neq\emptyset$, 
then $gB$ lies in the $3r$-neighborhood of $A$, and since
the $G$-action on $M$ is Palais-proper, so is the $G$-action on $\X(M)$.

(3) is proved in~\cite[Theorem 4.3.4]{Pal-slice} assuming (1)--(2) with the metric given by
${\bar d}_D(Gx, Gy)=\inf\{d_D(x,gy)\,:\,g\in G\}$.
\end{proof}

\begin{rmk}
\label{rmk: properties of Xk(M)}
The proof of Lemma~\ref{lem: properties of XkD(M)} works for $\X(M)$ in place of $\X_D(M)$
by replacing  $d_D$ with $d(A,B)=\min(1, d_D(A,B))$ if $A$, $B\in\X_D(M)$ for some $D$, 
and $d(A,B)=1$ otherwise. 
\end{rmk}

\begin{rmk}
\label{rmk: properties of Xk0M)}
The proof of Lemma~\ref{lem: properties of XkD(M)} goes through as written
if we give $\X(M)$ the $C^k$ topology, where $k$ is any nonnegative integer.
\end{rmk}

\begin{rmk}
\label{rmk: no Hausdorff}
Let $G$ be any group of isometries of $\R^2$ that contains a reflection $r$
in the $y$-axis, as well as a (nontrivial) 
rotation about the origin. 
Let us justify the claim made before Question~\ref{quest: intro C0 top} that 
there is no Hausdorff continuous $G$-equivariant center defined on the set of 
$2$-disks in $\X(\R^2)$.
The unit disk $D^2\subset\R^2$ is Hausdorff close
to a smoothly embedded  $2$-disk $D$ obtained by removing
from $D^2$ a small $r$-invariant
neighborhood of the segment $\{(0, y)\in D^2\,:\, y\ge -\frac{1}{2}\}$. 
Since $G$ contains $r$, any $G$-equivariant center of $D$ is contained in
$\{(0, y)\in D^2\,:\, y<-\frac{1}{2}\}$, while the $G$-invariant center
of $D^2$ is the origin because $G$ contains a rotation.
\end{rmk}

\section{Proof of Theorem~\ref{thm: intro main}}
\label{sec: the proof}

In this section $M$, $G$ are as in Section~\ref{sec: intro}. 
Thus  $M^G$ is a smooth submanifold of $M$, and if $D$
is a compact subset of $M$, then $G_D$ is a compact subgroup of $G$. Also
$D^{G_D}=D\cap M^{G_D}$.
By {\em small\,} we mean close to the identity.

\begin{lem}
\label{lem: rel Palais}
Any $D\in\X(M)$ has a compatible smooth
atlas $\a$ such that the given $G_D$-action on 
$(D,\a)$ can be  $C^1$ approximated by a smooth $G_D$-action,
and moreover, these two actions are conjugate by a $C^1$
diffeomorphism that is $C^1$ close to the identity.
\end{lem}
\begin{proof}
The existence of a compatible smooth atlas is proved in~\cite[Theorem 3.6, Chapter 2]{Hir-book}.
If $D$ is a closed manifold, the other claims are contained 
in~\cite[Theorem C]{Pal-C1-equiv-Cinfty}. If $D$ has boundary,
we first double the given $C^1$ actions of $G_D$
along $\d D$. To equivariantly
smooth the double at $\d D$ we may have to adjust the action, 
conjugating it by a small $C^1$ isotopy near $\d D$. Then we apply~\cite[Theorem C]{Pal-C1-equiv-Cinfty}
to the double to approximate the $C^1$ action by a $C^1$ equivalent smooth action.
This action stabilizes a copy of $D$ that is $C^1$ close to $D$, so after a conjugation 
by a small smooth diffeomorphism it can be made to stabilize $D$ inside the double. 
\end{proof}

\begin{lem} 
\label{lem: fixed point set is C1}
Let $D\in\X(M)$ and suppose $D^{G_D}\neq\emptyset$. Then
$D^{G_D}$ is a properly embedded $C^1$ submanifold of 
$D$ that intersects the relative interior
of $D$ and is transverse to $\d D$. 
\end{lem}
\begin{proof}
In the compatible atlas of Lemma~\ref{lem: rel Palais} the $G_D$-action on $D$ is 
near a smooth action, and the two actions are conjugate by a small $C^1$ diffeomorphism.
Let $\phi$ denote the smooth $G_D$-action on the double of $D$, and fix a Riemannian metric invariant
under $\phi$. Its fixed point set  is a smooth compact boundaryless submanifold 
which is totally geodesic in the metric. 
Since $\d D$ is $\phi$-invariant if $\phi$ fixes $x\in \d D$, then it also fixes every 
point in a convex neighborhood of $x$
that lies on the geodesic through $x$ that is orthogonal to $\d D$.
By assumption $G_D$ fixes a point of $D$, and hence so does $\phi$. It follows that
the fixed point of $\phi$  intersects the relative interior
of $D$ and is transverse to $\d D$. Hence the same is true for the given $G_D$ action
on the double, and the claim follows.
\end{proof}

\begin{rmk}
Lemma~\ref{lem: fixed point set is C1} is immediate by transversality
if $\dim(D)=\dim(M)$. In general, the intersection of $\d D$ and $D^{G_D}$
is not transverse, e.g., a round disk $D$ in the $xy$-plane in $\R^3$
is invariant under  rotation by $\pi$ about the $x$-axis,
which would be disjoint from $\d D$ if they met transversely.
\end{rmk}

We partition $\X(M)$ as $\bigcup_{k, l}\D^{k, l}$ where $\D^{k,l}(M)$ is the set of
all $D\in\X(M)$ such that $D^{G_D}$ is a $k$-dimensional manifold with $l$ connected components.

\begin{rmk}
Even when $D$ is a disk, the manifold $D^{G_D}$ may have any finite number of components, e.g.,
if $G$ is connected and nonabelian and $F$ is any finite CW complex,
then there is a smooth $G$-action on some high-dimensional disk $D$
such that $D^{G_D}$ is homotopy equivalent to $F$, see~\cite[Theorems 3 and 5]{Oli-MathZ}.
On the other hand, we shall see in Appendix~\ref{sec: actions on D4} that if $D$ is a disk of dimension $\le 4$,  
then $D^{G_D}$ is acyclic, and in particular, connected.
\end{rmk}

Henceforth, we need the following notion of a tubular neighborhood of a $C^1$ submanifold $N$ 
of a Riemannian manifold $M$. By~\cite[Thm. 3.6, Chap 2]{Hir-book} 
there is a $C^1$ diffeomorphism $\phi\colon M\to M$ such that 
$\widetilde N:=\phi(N)$ is a $C^\infty$ submanifold of $M$. 
If $\widetilde U$ is a Riemannian tubular neighborhood of $\widetilde N$ 
(produced via the normal exponential map), then
we call $U:=\phi^{-1}(\widetilde U)$ a {\em tubular neighborhood\,} of $N$, and 
furthermore, if $\tilde\pi\colon \widetilde U\to \widetilde N$ is the nearest point projection,
we call $\pi:=\phi^{-1}\circ\widetilde\pi\circ\phi\colon U\to N$ 
the \emph{normal bundle projection} of $U$.

\begin{lem} 
\label{lem: fixed pt set contin}
The map $\D^{k,l}(M)\to\X(M)$
given by $D\to D^{G_D}$ is continuous.
\end{lem}
\begin{proof}
Fix any $J\in \D^{k,l}(M)$ and prove continuity at $J$. By Lemma~\ref{lem: implications of slice}(vi)
there is a neighborhood $U$ of $J$ in $\D^{k,l}(M)$
such that $G_D$ can be conjugated into $G_J$ by a small
element of $G$, and since such conjugation results in only small $C^1$ changes of $D$ and $D^{G_D}$
we may assume that $G_D\le G_J$. Thus both $D$ and $J$ are invariant under $H=G_D$.

After possibly shrinking $U$ further we can find a $C^1$ diffeomorphism $\psi$
with $\psi(D)=J$ such that $\psi$ is supported in a small neighborhood of $J$.
(To find $\psi$ we first isotope $\d D$ to $\d J$ in the the tubular neighborhood $\d J$,
then extend $J$ to a boundaryless embedded $C^1$ submanifold, use
its tubular neighborhood to isotope $D$ to $J$ relative boundary, 
and finally extend the isotopy to an ambient one). 
Thus $\psi H\psi^{-1}$ and $H$ are two $C^1$ close $H$-actions on $J$. 
By Lemma~\ref{lem: rel Palais} these actions are $C^1$ diffeomorphic
via a small diffeomorphism.

By Lemma~\ref{lem: fixed point set is C1}
the fixed point sets of the actions of 
$\psi H\psi^{-1}$, $H$, $G_J$ on $J$
are compact properly embedded $k$-dimensional $C^1$ submanifolds of $J$ with
$l$ connected components. 
Any proper embedding of compact $k$-manifolds with the same number of components
is surjective, so since $H\le G_J$ the fixed point sets of $H$, $G_J$ in $J$
coincide.  
Thus $J^{G_J}$ is $C^1$ close to the fixed point set of $\psi H\psi^{-1}$
which is $D^{G_D}$.
\end{proof}

\begin{lem}
\label{lem: tub nbhd straight}
If $B$ is a compact $C^1$ submanifold of a manifold $F$, then there are
two nested tubular neighborhoods $T\subset T_2$ of $\d B$ in $F$ and 
a $C^1$ self-map $q_B$ of $F$ that is the identity outside $T_2$, 
and that equals the normal bundle projection on $T$. 
\end{lem}
\begin{proof}
Fix a tubular neighborhood $T_2$ of $\d B$ in $F$,
and identify it with $[-2,2]\times \d B$
where $\d B$ corresponds to $\{0\}\times\d B$.
Let $T$ be the subset of $T_2$ corresponding under the identification
to $[-1,1]\times \d B$.  
Fix a smooth non-decreasing self-map $\tau$ of $[-2,2]$ that vanishes on $[-1,1]$ and
equals the identity near the endpoints $-2$, $2$. 
Then the self-map of 
$[-2,2]\times \d B$ given by $(t, z)\to (\tau(t), z)$ 
interpolates between the identity near the boundary and
the projection $[-1,1]\times\d B\to\{0\}\times \d B$.
\end{proof}

Let $\E^{k,l}(M)$ be the subspace of $\D^{k,l}(M)\times M$ consisting of all the
pairs $(D,u)$ with $D\in\D^{k,l}(M)$ and $u\in D^{G_D}$. Let 
$\pi\co\E^{k,l}(M)\to\D^{k,l}(M)$ be the coordinate projection, i.e., 
$\pi(D, u)=D$.

\begin{lem}
\label{lem: loc trivial}
The map $\pi\co \E^{k,l}(M)\to\D^{k,l}(M)$ 
is a locally trivial  bundle.
\end{lem}
\begin{proof}
Fix $J\in\D^{k,l}(M)$, set $B=J^{G_J}$, and extend $B$ to a $k$-dimensional boundaryless $C^1$ submanifold $F$ of $M$.
Let $q_F$ be a projection of a tubular neighborhood of $F$ in $M$.
For such $B$, $F$ let $q_B$ and $T$ be as in Lemma~\ref{lem: tub nbhd straight}. Let $q=q_B\circ q_F$. 
Using Lemma~\ref{lem: fixed pt set contin} we can assume that $D$ is so close to $J$ that
$q_F(D^{G_D})\subset T$ and $q_B$ restricts to a $C^1$ diffeomorphism of $q_F(D^{G_D})$
onto $B$. The map $(D, u)\to (D, q(u))$ is the desired 
local trivialization where $u\in D^{G_D}$.
\end{proof}

\begin{rmk}
The homeomorphism type of the fiber $\pi^{-1}(D)$ may 
depend on $D$. 
\end{rmk}

The group $G$ acts on $\E^{k,l}(M)$ diagonally, 
i.e., $g(D,u)=(gD, gu)$. 
In the following commutative diagram 
\begin{equation}
\label{form: cd}
\xymatrix{
\E^{k,l}(M) \ar[r]^\pi
\ar[d]_{p} &  \D^{k,l}(M)\ar[d]^{\bar p}  \\
\E^{k,l}(M)/G\ar[r]_{\bar \pi}& \D^{k,l}(M)/G}
\end{equation}
the vertical arrows are the $G$-orbit maps, and $\bar\pi$
sends the $G$-orbit of $(D,u)$ 
to the $G$-orbit of $D$.

By Remark~\ref{rmk: properties of Xk(M)} the $G$-spaces
$\D^{k,l}(M)$ and $\E^{k,l}(M)$ are Palais-proper
because the property is clearly inherited by invariant subspaces and 
preimages under equivariant maps. Also $\E^{k,l}(M)$
has a $G$-invariant metric induced by the $G$-invariant metrics on the factors
of $\D^{k,l}(M)\times M$.
The following result is a key observation of this paper, which reduces to 
Lemma~\ref{lem: loc trivial} when $G$ is trivial.

\begin{lem}
\label{lem: local triviality quotient}
$\bar \pi\co\E^{k,l}(M)/G\to\D^{k,l}(M)/G$ is a locally trivial fiber bundle.
\end{lem}
\begin{proof}
Let us first sketch the proof. The $G$-action permutes the fibers of
$\pi$ which are of the form $\pi^{-1}(y)=\{(y,u)\,:\, u\in y^{G_y}\}$.
Moreover, each fiber projects homeomorphically to $\E^{k,l}(M)/G$ because
the fiber is compact and the projection $p$ is injective on it:
If $g(y,u)=(y,v)$ with $u,v\in y^{G_y}$, then $g\in G_y$ and $gu=v$. 
Hence $gu=u$ implies $u=v$. To establish local triviality 
we analyze the structure of the orbit spaces via local slices.
Using that $G_y\le G_x$ for every $y$ in a $G_x$-slice
we show that the local trivialization of Lemma~\ref{lem: loc trivial} 
can be made $G_x$-equivariant, so it passes to the $G_x$-quotients of the slices,
yielding a local trivialization of $\bar\pi$.

Let us make this sketch rigorous. 
Set $X=\D^{k,l}(M)$. Fix $x\in X$ and set $H=G_x$. 
Let $S$ be a $H$-slice at $x$. 
Thus $G$ fixes $\pi^{-1}(x)$ pointwise, and hence, 
Lemma~\ref{lem: preim of slice} implies that
$E=\pi^{-1}(S)$ is an $H$-slice at any point of $\pi^{-1}(x)$.

Set $\bar E=p(E)$ and $\bar S=\bar p(S)$. Since $S$ is a slice, $\bar S$
is a neighborhood of $\bar p(x)$, and it suffices to show that
$\bar\pi$ is a locally trivial bundle over $\bar S$.

Let us show that $\bar E={\bar\pi}^{-1}(\bar S)$.
One inclusion follows from a diagram chase: 
If $z\in p(\pi^{-1}(S))$, then ${\bar\pi}(z)\in ({\bar\pi}\circ p)(\pi^{-1}(S))=
({\bar p}\circ\pi)(\pi^{-1}(S))=p(S)$.
Conversely, if $z\in {\bar\pi}^{-1}(\bar S)$, then ${\bar\pi}(z)=Gy$ for some $y\in S$.
Also $z$ is the $G$-orbit of some 
$(a, v)$ with $a\in X$, $v\in a^{G_a}$, and hence
${\bar\pi}(z)$ is the $G$-orbit of $\pi(a,v)=a$. Thus $a=gy$ for some $g\in G$.
Hence $z$ is the $G$-orbit of $(y, g^{-1}v)$, and moreover,
$g^{-1}v\in y^{G_y}$.  Thus $(y, g^{-1}v)\in E$, and so $z\in \bar E$.
 
Since $E$ and $S$ are $H$-slices,  the inclusion induced maps $E/H\to \bar E$,
$S/H\to \bar S$ are homeomorphisms, so we can identify
the map $\bar\pi\co \bar E\to\bar S$
with the map $E/H\to S/H$ also induced by $\pi$.

Set $B=x^H$ and note that $B$ is contained in $M^H$, which is
a boundaryless properly embedded smooth submanifold of $M$.
We can extend $B$ to a $k$-dimensional boundaryless $C^1$ submanifold $F$ of $M^H$.
Fix an $H$-invariant tubular neighborhood of $F$ in $M$, and let
$q_F$ be an $H$-equivariant projection of the neighborhood onto $M$ 
(obtained, e.g., as the nearest point projection 
of some smooth $H$-invariant Riemannian metric on $M$).

For such $B$, $F$ let $q_B$ and $T$ be as in Lemma~\ref{lem: tub nbhd straight}. 
Now Lemma~\ref{lem: fixed pt set contin} implies that by making $S$ smaller
we can assume that any $y\in S$ satisfies
$q_F(y^{G_y})\subset T$ and $q_B$ restricts to a $C^1$ diffeomorphism of $q_F(y^{G_y})$
onto $B$.  (Lemma~\ref{lem: implications of slice}(v) allows us to make $S$
arbitrary small.)
Then for each $y\in S$ the composite 
$q=q_B\circ q_F$ restricts to a $C^1$-diffeomorphism of 
$y^{G_y}$ onto $B$.

For every $y\in S$ we have $G_y\le H$, and hence $M^{G_y}\supseteq M^H\supseteq F$.
In particular, if $z\in y^{G_y}$ and $y\in S$, then $q(hz)=q(z)$
because $q$ is $H$-equivariant and $H$ acts trivially on the image of $q$.

The map $\phi\co E\to S\times B$ given by $(y, u)\to (y, q(u))$
is $H$-equivariant because it sends
$h(y,u)=(hy, hu)$ to $(hy, q(hu))=(hy, q(u))=(hy, hq(u))=h(y, q(y))$,
Also $\phi$ is a homeomorphism whose inverse sends $(y, v)$ to
$(y, (q\vert_{y^{G_y}})^{-1}(v))$.
Note that $H$ acts trivially on the $B$-factor. 
So $\phi$ descends to a homeomorphism  of the $H$-quotients 
$E/H\to S/H\times B$ which gives a desired local trivialization.
\end{proof}

Let $\D^k(M)$ be the subspace of $D\in\X(M)$ such that $D^{G_D}$ is contractible
and $k$-dimensional. 
Set $\E^k(M)=\{(D,u)\,:\, D\in\D^k(M), u\in D^{G_D}\}$. 
Note that $\D^k(M)$, $\E^k(M)$ are $G$-invariant subsets of $\D^{k,1}(M)$, $\E^{k,1}(M)$, respectively,
and moreover, $\E^k(M)$ is the $\pi$-preimage of $\D^k(M)$. Thus $\pi\co \E^k(M)\to\D^k(M)$ 
and ${\bar\pi}\co\E^k(M)/G\to\D^k(M)/G$ are locally trivial fiber bundles.

\begin{lem}
\label{lem: equiv extension}
If $A$ is a closed $G$-invariant subset of $\D^{k}(M)$, then
every continuous \mbox{$G$-equivariant} section $s_A\co A\to\E^{k}(M)$ of $\pi$ over $A$ 
extends to a continuous \mbox{$G$-equivariant} section $s$ of $\pi$ over $\D^{k}(M)$.
Moreover, if $s_A$ takes values in the interiors of the fibers,
then one can choose $s$ with the same property.
\end{lem}
\begin{proof}
By $G$-equivariance, $s_A$ descends to a continuous map ${\bar s}_A\co A/G\to\E^{k,l}(M)/G$ that takes
the $G$-orbit of $x\in A$ to the $G$-orbit of $s(x)$, i.e., 
${\bar s}_A(\bar p(x))=p(s_A(x))$. 
The following section extension property can be found 
in~\cite[Theorem 9]{Pal-inf-dim}:
Given any locally trivial fiber bundle whose 
base is metrizable and fiber is an absolute retract,
any section of the bundle
defined on a closed subset can be extended to the whole base. 
In our case the fiber is either a compact contractible manifold with boundary,
or its interior, which are absolute retracts.

By definition of the quotient topology, $A/G$ is closed in $\D^{k,l}(M)/G$.
Thus ${\bar s}_A$ extends to a continuous section $\bar s$ of $\bar \pi$.
If $s_A$ takes values in the interiors of the fibers,
then so does ${\bar s}_A$, and hence it can be extended to $\bar s$ 
with the same property.
Now define $s\co\D^{k,l}(M)\to\E^{k,l}(M)$ by letting $s(x)$ be the intersection of 
the fiber $\pi^{-1}(x)$ and the $G$-orbit $p^{-1}({\bar s}(\bar p(x)))$, see the diagram (\ref{form: cd}).  
Let us check that $s$ has the claimed properties. 

The intersection consists of a single point, indeed,
if $z$, $gz$ are in the intersection,
then $\pi(z)=x=\pi(gz)=g\pi(z)$ and hence $g$ is in the isotropy subgroup of $z$, 
which fixes $\pi^{-1}(\pi(z))$ pointwise so that $gz=z$.
If $a\in A$, then $s(a)$ and $s_A(a)$ are points of $\pi^{-1}(a)$ that are
mapped by $p$ to ${\bar s}({\bar p}(a))={\bar s}_A({\bar p}(a))$, and since such a point
is unique, $s$ extends $s_A$.

By construction $\pi(s(x))\subseteq\pi\left(\pi^{-1}(x))\right)=\{x\}$, so $s$ is a section.
The map $s$ is $G$-equivariant because $s(gx)$ is the intersection of 
$\pi^{-1}(gx)=g\pi^{-1}(x)$ and  $p^{-1}({\bar s}(\bar p(gx)))= p^{-1}({\bar s}(\bar p(x)))=
g\left(p^{-1}({\bar s}(\bar p(x)))\right)$,
which equals $gs(x)$.

We write $u\approx v$ to indicate that $u,v$ are close to each other.
To prove continuity of $s$ write $s(x)=(x,c)$ and $s(y)=(y,d)$, where $c\in x$, $d\in y$,
assume $x\approx y$, and try to show that $c\approx d$.
From $x\approx y$ we get closeness of the the $G$-orbits 
$p^{-1}({\bar s}(\bar p(x))$, $p^{-1}({\bar s}(\bar p(y))$ of
$s(x)$, $s(y)$, respectively. Hence there is $g\in G$ such that $gs(x)\approx s(y)$,
or equivalently, $gx\approx y$ and $gc\approx d$. 
From $x\approx y$ we get $gx\approx gy$, and hence
$x\approx y\approx gx\approx gy$. Then $g$ is close to some elements of the isotropy subgroups
of $x$, $y$, which fix the fibers $\pi^{-1}(x)$, $\pi^{-1}(y)$ pointwise.
Hence $c\approx gc$ and $d\approx gd$ which together with $gc\approx d$ implies
$c\approx d$.
\end{proof}

The above proof also yields the following observation: 
\smallskip

\begin{lem}
\label{lem: equiv extension from base}
If $A$ is a closed $G$-invariant subset of $\D^{k}(M)$, then
every \mbox{$G$-equivariant} section ${\bar s}_A\co A/G\to\E^{k,l}(M)/G$
extends to a continuous section $\bar s$ of $\bar\pi$ over $\D^{k}(M)/G$
which lifts to \mbox{$G$-equivariant} section of $\pi$ over $\D^k(M)$. 
Moreover, if ${\bar s}_A$ takes values in the interiors of the fibers,
then one can choose $s$ with the same property.
\end{lem}

Set $\D(M)=\bigcup_{k=0}^n\D^k(M)$. We are now ready to prove Theorem~\ref{thm: intro main}. 
\medskip

\begin{thm}
For any $G$-invariant closed subset $Z$ of $\X(M)$ every $G$-equivariant center
$\mathfrak z\co Z\to M$ extends to a $G$-equivariant center $\cf\co\D(M)\to M$.
\end{thm}
\begin{proof}
Given a nonnegative integer $m$ let $\Se^m=\bigcup_{i\le m}\D^i$. 
Note that $\Se^m$ is closed in $\D(M)$ for if $D_i\to D$,
then $G_D$ contains a conjugate of $G_{D_i}$ by Lemma~\ref{lem: implications of slice}(vi) 
and, in particular, 
if each $D_i$ lies in $\Se^{m}$, then so does $D$.

We proceed by induction on the dimension of $D^{G_D}$.
Let $l$ be the smallest integer with $\Se^l\neq\emptyset$. Use 
Lemma~\ref{lem: equiv extension from base} with $A=Z\cap \Se^l$
to define a $G$-equivariant center $\cf\co\Se^l\to M$ that extends $\mathfrak z$.
Inductively, we suppose that a map $\cf$ with claimed properties
is defined on $\Se^k$, and hence on $\Ze^k=\Se^k\cup (Z\cap \Se^{k+1})$, and then
try to extend it to $\Se^{k+1}$.

To this end 
fix a smooth $G$-equivariant embedding of $\tau\co M\to\mathcal H$ 
where $\mathcal H$ is a Hilbert space  
equipped with the orthogonal $G$-action~\cite[Theorem 0.1]{Kaa-emb}.
The latter means that each element of $g$ is continuous, linear,
and preserves the inner product. 
Then the usual Riemannian open tubular neighborhood $T_\tau$ of $\tau$ is $G$-invariant 
and its projection $p_\tau\co T_\tau\to M$ is $G$-equivariant~\cite[Theorem 5.1]{Kaa-tub-nbhd}. 

In~\cite[Corollary 3.5 and Example 4.1]{Ant-eq-maps-cnvx-str} 
one finds a general extension result for $G$-maps from a Palais-proper space $X$ with 
a $G$-invariant metric to any locally convex linear $G$-space, such as $\mathcal H$.
It implies that for any closed $G$-invariant subset $A$ of $X$
every $G$-map $A\to\mathcal H$ can be extended
to a $G$-map $X\to \mathcal H$.
For this result the $G$-action on $\mathcal H$ must be linear, continuous, 
and every compact subgroup of $G$ has to fix a point of $\mathcal H$. 
The latter is true for orthogonal actions of compact groups on $\mathcal H$, and
in fact, any isometric group action with a bounded
orbit on a $\mathrm{CAT}(0)$ space fixes the circumcenter of the orbit, and 
in a $\mathrm{CAT}(0)$ space, such as $\mathcal H$, every bounded set has a unique 
circumcenter~\cite[Proposition II.2.7]{BriHae-book}.

Thus $\cf\co \Ze^k\to M\subset\mathcal H$ 
can be extended to a $G$-map from $\tilde{\cf}\co \Se^{k+1}\to \mathcal H$.
Set $O_\tau=\tilde\cf^{\,-1}(T_\tau)$ and $\bar c=p_\tau\circ\tilde\cf\vert_{O_\tau}$. 
Thus $\bar\cf\co O_\tau\to M$ is a $G$-map that extends $\cf$.

Let us show that there is a closed neighborhood $V$ of $\Ze^k$ in $O_\tau$ 
on which $\bar c(D)\in\Int(D)$. To this end
fix a $G$-invariant metric $d$ on $\D$ (using the metric of Remark~\ref{rmk: properties of Xk(M)}
in each $\D^k(M)$ and setting the distance between points of different $\D^k(M)$ to be $1$). 
Since $\bar{\cf}$ is continuous for any $D\in\D$ with $\bar{\cf}(D)\in\Int(D)$ 
there is $\e_D>0$ such that $\bar{\cf}(I)\in\Int(I)$ for every $I$ in the $\e_D$-ball
centered at $D$.
We can use the same $\e_D$ for every $D$  in the same $G$-orbit. 
The intersection with $\Se^{k+1}$ of the union of all such balls is a
$G$-invariant neighborhood $U$ of $\Ze^k$ in $\Se^{k+1}$, 
and $\bar\cf\vert_U (D)\in\Int(D)$.
Finally, $U$ contains a closed $G$-invariant neighborhood $V$ of $\Ze^k$, e.g., let
$V$ be the $f$-preimage of $\left[0,\frac12\right]$ where 
\[
f(x)=\frac{d(x, \Ze^k)}{d(x, \Ze^k)+d(x, \Se^{k+1}\setminus U)}.\]

It remains to enlarge $V$ to $\Se^{k+1}$.
The restriction of $\bar\cf$ to $V_0=V\cap\D^{k+1}$
is a $G$-equivariant section of the bundle $\E^{k+1}\to\D^{k+1}$ over $V_0$. 
Applying Lemma~\ref{lem: equiv extension} to the bundle whose fibers are $\Int(D^{G_D})$,
we extend $\bar\cf\vert_{V_0}$ to a $G$-equivariant
section $\cf_0$ over $\D^{k+1}$.
Since $\D^{k+1}$, $\Int(V)$ are open in $\Se^{k+1}$
and $\cf_0$, $\bar{\cf}\vert_{\Int(V)}$ agree on $\D^{k+1}\cap\Int(V)$, the maps
define a $G$-map $\Se^{k+1}\to M$ that takes each $D$
to a point in $\Int(D)$. 
This completes the induction step and proves the theorem.
\end{proof}

\begin{rmk}
The above proof reveals that Theorem~\ref{thm: intro main}
is a formal consequence of Lemma~\ref{lem: loc trivial}, and hence
if the lemma is true for the $C^0$ topology on $\X(M)$,
then the answer to Question~\ref{quest: intro C0 top} is affirmative.
\end{rmk}

\begin{proof}[Proof of Proposition~\ref{prop: no center}]
As discussed before the statement of the corollary, there exists 
an orthogonal $A_5$-action on $\R^r$ that has no fixed
points on an $A_5$-invariant smoothly embedded $6$-disk in $\R^r$.
Denote the $6$-disk by $\Delta$, consider 
the product of $\Delta$ and the $(n-6)$-disk with the trivial \mbox{$A_5$-action,} 
and smooth corners. The result is a smoothly embedded $A_5$-invariant
$n$-disk in $\R^{n+r-6}$ on which $A_5$ acts without fixed points.
Composing with the standard inclusion $\R^{n+r-6}\subset \R^m$ we can think
of the $n$-disk as sitting in $\R^m$, where we let 
$A_5$ act trivially on the orthogonal complement of $\R^{r+n-6}$ in $\R^m$. 
This gives the desired claim when $m\ge n+r-6\ge r$.

For the other case consider the inclusion $\R^r\subset\R^n$, where
$A_5$ acts trivially on the orthogonal complement of $\R^r$ in $\R^n$.
Let $\Delta^\prime$ be an $A_5$-invariant smooth tubular neighborhood of
$\Delta$ in $\R^n$ on which the $A_5$-action is fixed point free.
Note that $\Delta^\prime$ is an embedded $n$-disk in $\R^n$, so
composing with the inclusion $\R^n\subset\R^m$
gives the claim when $m\ge n\ge r$. 
\end{proof}

\appendix

\section{Compact Lie Group Actions on Low Dimensional Disks}
\label{sec: actions on D4}
  
The results of this appendix are surely known to experts but we could not find them in the literature.

It was shown by B.~Ker{\'e}kj{\'a}rt{\'o}
that any compact topological group action on a $2$-disk 
is equivalent to a linear action, see~\cite{Kol} and references therein. 
More precisely, cone off the boundary of the $G$-action on $D^2$, use~\cite{Kol}
to conclude that the resulting action on $S^2$ is topologically equivalent to 
an action by a subgroup of $O(3)$, which is actually in $O(2)$ since it has a fixed point,
and finally restrict the equivalence to $D^2$. 

For a smooth compact Lie group action, the uniformization gives  an alternative route:  
Find a $G$-invariant Riemannian metric on the interior, and map it conformally to the standard
hyperbolic disk so that the Lie group becomes a compact group of hyperbolic isometries,
and hence after conjugation the standard $O(2)$.
The same reasoning gives linearity of smooth actions on the closed interval,
whose isometry group in any metric is $O(1)$. 

It was shown in~\cite[Theorem B]{KwaSch-S3-linear} that any smooth action
of a compact Lie group on the $3$-disk is smoothly equivalent to a linear action.

\begin{lem}\label{lem:4-disk1}
If $G$ is a nontrivial finite group that
acts smoothly and effectively on the $4$-disk $D$, 
then either $D^G$ is a disk of dimension $\le 2$, or $G$ is an order two group
generated by an orientation-reversing involution and 
$D^G$ is a compact acyclic $3$-manifold.
\end{lem}
\begin{proof}
Any smooth finite group action on a $4$-disk has a fixed point~\cite[Theorem II.2]{BKS-fixed-pt-D4}.
Thus $D^G$ is a compact smooth properly embedded submanifold of $D$.
Let $S$ be the double of $D$ along $\d D$ equipped with the smooth $G$-action.
Thus $S$ is a smooth homotopy $4$-sphere and $S^G$ is nonempty.

The fixed point set of any smooth (or, more generally, locally linear) 
orientation-preserving action of a finite group on a homology $4$-sphere
is either empty or homeomorphic to a sphere~\cite[Theorem 2.1]{Dem}.
Let $H\le G$ be the subgroup of the orientation-preserving elements.
The group $G/H$ has order at most two,  because
$H$ is  the kernel of the $G$-action on $H_4(D, \d D;\Z)\cong \Z$.

Thus $S^H$ is a closed smooth submanifold of $S$ that is homeomorphic to
a sphere of dimension $l\in [0,4]$. 
Since $H$ preserves orientation, we have $l\neq 3$ else
$H$ would have to act nontrivially on a one-dimensional fiber of
the normal bundle to $D^G$.
If $l\le 2$, then $D^H$ is homeomorphic to the $l$-disk, the only manifold whose double is the $l$-sphere.
Since $D^G$ is the fixed point set of the $G/H$-action on $D^H$
we conclude for $l\le 2$ that $G/H$ acts linearly on $D^G$, and hence $D^G$ is 
a subdisk of $D^H$ as claimed. 

It remains to consider the case $l=4$ where $H$ is trivial and 
$G$ is generated by an orientation-reversing involution of $D$. 
Then $D^G$ is a homology disk~\cite[Theorem III.5.2]{Bre-book} of dimension $l\in [0,3]$. 
The $G$-action in a fiber of the normal bundle to $D^G$
only fixes the origin and reverses the orientation, hence $l$ is odd.
If $l=1$, then $D^G$ is a one-dimensional disk, as claimed, and we are left
with the exceptional case $l=3$ where $D^H$ is acyclic.
\end{proof}

\begin{rmk}
Any compact acyclic $3$-dimensional smooth submanifold of $S^3$ is homeomorphic to $D^3$
because its boundary is a homology $2$-sphere, and hence is homeomorphic to $S^2$,
so that the Schoenflies theorem applies~\cite{Bro-sch}. 
In particular, if $M$ and $G$ are as in Corollary~\textup{\ref{cor: 4-disk main}}
and $D$ is a $C^1$ embedded $4$-disk in $M$, then
$D^{G_D}$ is homeomorphic to a disk.
\end{rmk}

\begin{lem}\label{lem:4-disk2}
If $G$ is an infinite compact Lie group that
acts smoothly and effectively on the $4$-disk $D$, then $D^G$
is a disk of dimension $\le 2$. 
\end{lem}
\begin{proof}
If the claim is true for the connected component $G_0$ of $G$, then
it is true for $G$ because if $G_0$ fixes a disk $F$ of dimension $\le 2$,
then 
$G/G_0$ acts on $F$. Since any compact Lie group action on a disk of dimension $\le 2$
is equivalent to a linear action, $G/G_0$ fixes a subdisk of $F$ of
dimension $\le 2$, which is then fixed by $G$.

Thus we may assume that $G$ is connected. 
By the structure theory of compact Lie groups
there is a surjective homomorphism $T\times \bar G\to G$ with finite kernel
where $T$ is a torus and $\bar G$ is a semisimple connected compact Lie group. 
The resulting (possibly ineffective) $T$-action on $D^4$ has a 
 fixed point set $D^T$ which is an integral homology disk of dimension $0$, $2$, or $4$,
see~\cite[Theorem III.10.3]{Bre-book}, and in particular it is non-empty.  
If $T$ is nontrivial, then $D^T$ is a nowhere dense~\cite[Theorem III.9.5]{Bre-book} 
smooth submanifold of $D$, and so $\dim(D^T)<4$. 
Thus $D^T$ is diffeomorphic to a disk of dimension $\le 2$. Since $\bar G$ normalizes $T$, it 
acts on $D^T$, and the argument of the previous paragraph applies
to show that $G$ fixes a disk of dimension $\le 2$. 

Thus we may assume that $G$ is semisimple. The principal orbit $G/H$
cannot be a point because then $G$ fixes an open subset of $D$~\cite[Theorem IV.3.1]{Bre-book},
and cannot be a circle as, e.g., in the homotopy sequence
of the bundle $G\to G/H$ the group $\pi_1(G/H)$ sits between finite groups $\pi_1(G)$,
$\pi_0(H)$. Also $G/H$ is a closed manifold embedded in $D^4$, so its dimension is at most $3$.
Thus $G/H$ has codimension $1$ or $2$. Then~\cite[Theorem IV.8.1 and Theorem IV.8.5]{Bre-book}
imply that the $G$-action on $\Int(D)$ is equivalent to an orthogonal action $\R^4$,
so its fixed point set is a linear subspace $V$, and in particular is nonempty. 
Since $D^G$ is a smooth submanifold that is transverse to $\d D$,
we get $\Int(D^G)=V$. Thus
\[
0\le \dim(D^G)=\dim(V)\le\dim(\Int(D)/G)=4-\dim(G/H)\le 2
\]
and therefore $D^G$ is a disk of dimension $\le 2$.
\end{proof}

\begin{lem}
\label{lem: acyclic 4d}
There is a smooth involution on $D^4$ whose fixed point set
is an acyclic non-simply connected $3$-manifold with boundary
that is properly embedded in $D^4$. 
\end{lem}
\begin{proof}
In~\cite{Maz-contract-4d, Poe-contract-4d} one finds
a smooth compact contractible $4$-manifold $C$
such that $\d C$ is not simply-connected while the double $DC$
of $C$ along $\d C$ is diffeomorphic to $S^4$. Let $\i$
be the doubling involution. Removing a small $\i$-invariant
open ball $B$ centered at $\d C$ yields a $\i$-invariant copy of $D^4$ 
that is the double of $C\setminus B$ along $\d C\setminus B$.
Here $\i$ permutes two copies of $C\setminus B$ and fixes
$\d C\setminus B$ pointwise. Finally, 
since $\d C$ is a non-simply connected homology sphere, $\d C\setminus B$ is a compact
acyclic non-simply connected $3$-manifold.
\end{proof}

Here is an example of a $5$-disk $D$ in $\R^5$ 
such that $G_D$ is generated by a
linear involution and $D^{G_D}$ is not contractible.

\begin{lem}
\label{lem: acyclic 5d}
If $r$ is the reflection in the equator of $S^5$ and $k\in\{4,5\}$, then
$S^5$ contains a smoothly embedded $r$-invariant copy of $D^k$
which transversely intersects the equator in an acyclic non-simply connected 
$(k-1)$-manifold with boundary.
\end{lem}
\begin{proof}
In the notations of Lemma~\ref{lem: acyclic 4d}
consider a proper embedding of $C$ into $D^5$ (e.g., obtained
by isotoping $C\subset S^4=\d D^5$ into the interior of $D^5$ and concatenating
the result with the track of isotopy over $\d C$). Removing from $C$ a small
$r$-invariant open $4$-disk $B$ centered at a point of $\d C\subset\d D^5$ and 
then doubling along $\d D^5$ gives an $r$-invariant copy of $DC\setminus B$ of 
Lemma~\ref{lem: acyclic 4d} in $S^5$ that intersects the equator
along $\d C\setminus B$. This covers the case $k=4$. 
For the case $k=5$ set $I=[-1,1]$ and note that 
$DC\times I$ embeds as
an $r$-invariant tubular neighborhood of $DC$
in $S^5$, and after smoothing corners $(DC\setminus B)\times I$ is 
an $r$-invariant $5$-disk in $S^5$ that intersects the equator 
along $(C\setminus B)\times I$.
\end{proof}

\section*{Acknowledgments}
M.G. first learned about the problem of finding a center
for planar Jordan domains from Eugenio Calabi in 1995. Also we thank John Etnyre, Ralph Howard,
Dan Margalit, and Krzysztof Pawalowski for helpful discussions.

\bibliographystyle{amsalpha}
\bibliography{cr}

\providecommand{\bysame}{\leavevmode\hbox to3em{\hrulefill}\thinspace}
\providecommand{\MR}{\relax\ifhmode\unskip\space\fi MR }
\providecommand{\MRhref}[2]{%
  \href{http://www.ams.org/mathscinet-getitem?mr=#1}{#2}
}
\providecommand{\href}[2]{#2}
\begin{thebibliography}{KMT91}

\bibitem[Ant05]{Ant-eq-maps-cnvx-str}
S.~A. Antonyan, \emph{Extending equivariant maps into spaces with convex
  structure}, Topology Appl. \textbf{153} (2005), no.~2-3, 261--275.

\bibitem[BH99]{BriHae-book}
M.~R. Bridson and A.~Haefliger, \emph{Metric spaces of non-positive curvature},
  Grundlehren der Mathematischen Wissenschaften [Fundamental Principles of
  Mathematical Sciences], vol. 319, Springer-Verlag, Berlin, 1999.

\bibitem[BKS90]{BKS-fixed-pt-D4}
N.~P. Buchdahl, S.~Kwasik, and R.~Schultz, \emph{One fixed point actions on
  low-dimensional spheres}, Invent. Math. \textbf{102} (1990), no.~3, 633--662.

\bibitem[BLP05]{BLP}
M.~Boileau, B.~Leeb, and J.~Porti, \emph{Geometrization of 3-dimensional
  orbifolds}, Ann. of Math. (2) \textbf{162} (2005), no.~1, 195--290.

\bibitem[BM05]{BakMor}
A.~Bak and M.~Morimoto, \emph{The dimension of spheres with smooth one fixed
  point actions}, Forum Math. \textbf{17} (2005), no.~2, 199--216.

\bibitem[Bre72]{Bre-book}
G.~E. Bredon, \emph{Introduction to compact transformation groups}, Academic
  Press, New York-London, 1972, Pure and Applied Mathematics, Vol. 46.

\bibitem[Bro60]{Bro-sch}
M.~Brown, \emph{A proof of the generalized {S}choenflies theorem}, Bull. Amer.
  Math. Soc. \textbf{66} (1960), 74--76.

\bibitem[DL09]{DinLee-RF-geom}
J.~Dinkelbach and B.~Leeb, \emph{Equivariant {R}icci flow with surgery and
  applications to finite group actions on geometric 3-manifolds}, Geom. Topol.
  \textbf{13} (2009), no.~2, 1129--1173.

\bibitem[DM89]{Dem}
S.~De~Michelis, \emph{The fixed point set of a finite group action on a
  homology four sphere}, Enseign. Math. (2) \textbf{35} (1989), no.~1-2,
  107--116.

\bibitem[Hir94]{Hir-book}
M.~W. Hirsch, \emph{Differential topology}, Graduate Texts in Mathematics,
  vol.~33, Springer-Verlag, New York, 1994, Corrected reprint of the 1976
  original.

\bibitem[Kan94]{Kaa-emb}
M.~Kankaanrinta, \emph{On embeddings of proper smooth {$G$}-manifolds}, Math.
  Scand. \textbf{74} (1994), no.~2, 208--214.

\bibitem[Kan07]{Kaa-tub-nbhd}
\bysame, \emph{Equivariant collaring, tubular neighbourhood and gluing theorems
  for proper {L}ie group actions}, Algebr. Geom. Topol. \textbf{7} (2007),
  1--27.

\bibitem[KMT91]{KMT}
M.~J. Kaiser, T.~L. Morin, and T.~B. Trafalis, \emph{Centers and invariant
  points of convex bodies}, Applied geometry and discrete mathematics, DIMACS
  Ser. Discrete Math. Theoret. Comput. Sci., vol.~4, Amer. Math. Soc.,
  Providence, RI, 1991, pp.~367--385.

\bibitem[KN96]{KobNom}
S.~Kobayashi and K.~Nomizu, \emph{Foundations of differential geometry. {V}ol.
  {I}}, Wiley Classics Library, John Wiley \& Sons, Inc., New York, 1996,
  Reprint of the 1963 original, A Wiley-Interscience Publication.

\bibitem[Kol06]{Kol}
B.~Kolev, \emph{Sous-groupes compacts d'hom\'eomorphismes de la sph\`ere},
  Enseign. Math. (2) \textbf{52} (2006), no.~3-4, 193--214.

\bibitem[KS92]{KwaSch-S3-linear}
S.~Kwasik and R.~Schultz, \emph{Icosahedral group actions on {${\bf R}^3$}},
  Invent. Math. \textbf{108} (1992), no.~2, 385--402.

\bibitem[Maz61]{Maz-contract-4d}
B.~Mazur, \emph{A note on some contractible {$4$}-manifolds}, Ann. of Math. (2)
  \textbf{73} (1961), 221--228.

\bibitem[Mic56a]{michael1}
E.~Michael, \emph{Continuous selections. {I}}, Ann. of Math. (2) \textbf{63}
  (1956), 361--382.

\bibitem[Mic56b]{michael2}
\bysame, \emph{Continuous selections. {II}}, Ann. of Math. (2) \textbf{64}
  (1956), 562--580.

\bibitem[Mic59]{michael60}
\bysame, \emph{Paraconvex sets}, Math. Scand. \textbf{7} (1959), 372--376.

\bibitem[Mos06]{moszynska:book}
Maria Moszy\'nska, \emph{Selected topics in convex geometry}, Birkh\"auser
  Boston, Inc., Boston, MA, 2006, Translated and revised from the 2001 Polish
  original. \MR{2169492}

\bibitem[Oli76]{Oli-MathZ}
R.~Oliver, \emph{Smooth compact {L}ie group actions on disks}, Math. Z.
  \textbf{149} (1976), no.~1, 79--96.

\bibitem[Pal60]{Pal-mem}
R.~S. Palais, \emph{The classification of {$G$}-spaces}, Mem. Amer. Math. Soc.
  No. 36, 1960.

\bibitem[Pal61]{Pal-slice}
\bysame, \emph{On the existence of slices for actions of non-compact {L}ie
  groups}, Ann. of Math. (2) \textbf{73} (1961), 295--323.

\bibitem[Pal66]{Pal-inf-dim}
\bysame, \emph{Homotopy theory of infinite dimensional manifolds}, Topology
  \textbf{5} (1966), 1--16.

\bibitem[Pal70]{Pal-C1-equiv-Cinfty}
\bysame, \emph{{$C^{1}$} actions of compact {L}ie groups on compact manifolds
  are {$C^{1}$}-equivalent to {$C^{\infty }$} actions}, Amer. J. Math.
  \textbf{92} (1970), 748--760.

\bibitem[Pix76]{pixley1976}
C.~P. Pixley, \emph{Continuously choosing a retraction of a separable metric
  space onto each of its arcs}, Illinois J. Math. \textbf{20} (1976), no.~1,
  22--29.

\bibitem[Poe60]{Poe-contract-4d}
V.~Poenaru, \emph{Les decompositions de l'hypercube en produit topologique},
  Bull. Soc. Math. France \textbf{88} (1960), 113--129.

\bibitem[RS98]{RepSem-book}
D.~Repov{\v s} and P.~V. Semenov, \emph{Continuous selections of multivalued
  mappings}, Mathematics and its Applications, vol. 455, Kluwer Academic
  Publishers, Dordrecht, 1998.

\bibitem[RS02]{RepSem-II}
\bysame, \emph{Continuous selections of multivalued mappings}, Recent progress
  in general topology, {II}, North-Holland, Amsterdam, 2002, pp.~423--461.

\bibitem[RS14]{RepSem-III}
\bysame, \emph{Continuous selections of multivalued mappings}, Recent progress
  in general topology. {III}, Atlantis Press, Paris, 2014, pp.~711--749.

\bibitem[Sch14]{Sch-book}
R.~Schneider, \emph{Convex bodies: the {B}runn-{M}inkowski theory}, expanded
  ed., Encyclopedia of Mathematics and its Applications, vol. 151, Cambridge
  University Press, Cambridge, 2014.

\end{thebibliography}
\end{document}